\documentclass[reqno]{amsart}
\usepackage{fullpage}

\usepackage{amsfonts,amssymb, amsmath,graphicx}
\usepackage{oldgerm}
\usepackage{amscd,color}

\newtheorem{proposition}[equation]{Proposition}
\newtheorem{theorem}[equation]{Theorem}
\newtheorem{claim}[equation]{Claim}
\newtheorem{question}[equation]{Question}
\newtheorem{lemma}[equation]{Lemma}

\newtheorem{corollary}[equation]{Corollary}

\theoremstyle{remark}

\theoremstyle{definition}
\newtheorem{definition}[equation]{Definition}

\newtheorem{notation}[equation]{Notation}
\newtheorem{convention}[equation]{Convention}

\theoremstyle{remark}

\newtheorem{remark}[equation]{Remark}
\newtheorem{example}[equation]{Example}

\numberwithin{equation}{section}

\newcommand{\calD}{{\mathcal D}}
\newcommand{\calP}{{\mathcal P}}
\newcommand{\calM}{{\mathcal M}}
\newcommand{\bP}{{\mathbb P}}
\newcommand{\bZ}{{\mathbb Z}}

\newcommand{\bQ}{{\mathbb Q}}

\newcommand{\calC}{{\mathcal C}}

\DeclareMathOperator{\tr}{tr}
\DeclareMathOperator{\rank}{rank}
\DeclareMathOperator{\sign}{sign}
\DeclareMathOperator{\prim}{prim}

\begin{document}
 \title[Maximally Algebraic Cubics]{Maximally Algebraic potentially irrational Cubic Fourfolds}
\author[R. Laza]{Radu Laza}
\address{Stony Brook University,  Stony Brook, NY 11794, USA}
\email{radu.laza@stonybrook.edu}

\date{\today}
\thanks{Research of the author is supported in part by NSF grants DMS-1254812 and DMS-1361143}

\begin{abstract}
A well known conjecture asserts that a cubic fourfold $X$ whose transcendental cohomology $T_X$ can not be realized  as the transcendental cohomology of a $K3$ surface is irrational. Since the geometry of cubic fourfolds is intricately related to  the existence of  algebraic $2$-cycles on them, it 
  is natural to ask for the most algebraic cubic fourfolds $X$ to which this conjecture is still applicable. In this paper, we show that for an appropriate ``algebraicity index'' $\kappa_X\in \bQ_+$, there exists a unique  class of cubics maximizing $\kappa_X$, not having an associated $K3$ surface; namely, the cubic fourfolds with an Eckardt point (previously investigated in  \cite{LPZ}). Arguably, they are the most algebraic potentially irrational cubic fourfolds, and thus a good testing ground for the Harris, Hassett, Kuznetsov conjectures.  
\end{abstract}
\maketitle

\bibliographystyle{amsalpha}

\section{Introduction}
A celebrated result of Clemens and Griffiths \cite{CG} says that a smooth cubic threefold $Y$ is irrational. This follows by showing that while the Hodge structure on $H^3(Y)$ looks like (the Tate twist of) the Hodge structure of a curve, it is in fact not coming (over $\bZ$) from a curve. Similarly, for cubic fourfolds, one sees that the Hodge structure on $H^4(X)_{\prim}$ is of $K3$ type. 
In analogy with the Clemens--Griffiths result, it is natural to conjecture that if the transcendental cohomology of $X$ is not actually coming from a $K3$ surface $S$ (i.e. $H^4(X)_{\tr}\cong H^2(S)(-1)_{\tr}$), then $X$ is not rational. This was proposed by Harris, and some evidence was given by Hassett \cite{hassettT,hassettrational}: {\it all known rational cubic fourfolds have transcendental cohomology (over $\bZ$) induced from a $K3$ surface} (see also \cite{AT} for a discussion of the related Kuznetsov conjecture). However, we emphasize that {\bf no} example of irrational cubic fourfold (or for that matter any irrational cubic of dimension $> 3$)  is currently known. There are numerous papers on the subject of cubics and rationality, but the focus so far seems to  have been on (Hodge) general cubics or almost general cubics (e.g. cubics containing a plane).  In this paper, we turn the question around and ask: 
\begin{question}\label{questionmaxalg} Assuming that there exist irrational cubic fourfolds, which are the maximally algebraic cubic fourfolds $X$ for which the rationality fails?
\end{question} 

For a cubic $X$,  its algebraicity is naturally measured by the rank of the group of primitive algebraic cycles
\begin{equation}\label{defrho}
\rho_X:=\rank \left(H^{2,2}(X)\cap H^{4}(X,\bZ)_{\prim}\right)\in \{0,\dots, 20\}.
\end{equation}
For a general cubic fourfold $X$, $\rho_X=0$. In general,  $\rho_X=k$ is a codimension $k$ condition on moduli. More precisely, due to the theory of periods  for cubic fourfolds (\cite{voisin}, \cite{hassettspecial}, \cite{gitcubic,cubic4fold}, \cite{lcubic}), for $k\ge 1$, the locus of cubics with $\rho_X\ge k$ is a countable union of  codimension $k$ closed algebraic subvarieties in the moduli space $\calM$ of cubic fourfolds. These special loci can be indexed (compare \cite{hassettspecial}) by a second positive integer 
\begin{equation}\label{defd}
d_X:=\det \left(H^{2,2}(X)\cap H^{4}(X,\bZ)\right),
\end{equation}
(N.B. the associated locus might be  empty or reducible for some values of the pair $(\rho_X,d_X)$). It is natural to understand {\it maximally algebraic} cubic as both maximizing $\rho_X$ and minimizing $d_X$. We propose to combine $(\rho_X,d_X)$ in a single  value {\it algebraicity index} for $X$, defined by
\begin{equation}\label{defkappa}
\kappa_X=\frac{2^{\rho_X}}{d_X}\in \bQ_{+}.
\end{equation}
Thus, a more precise version of Question \eqref{questionmaxalg} is to ask to maximize $\kappa_X$ subject to the irrationality constraint. 

\begin{remark} Vinberg \cite{vinberg}  studied the two most algebraic $K3$ surfaces. Similar to the $K3$ case, the two most algebraic cubic fourfolds (i.e. those giving maximal values for $\kappa_X$) are the (rigid) cubics with transcendental lattices: $A_2(-1)=\left(\begin{matrix} -2&1\\1&-2\end{matrix}\right)$ and $(A_1(-1))^2=\left(\begin{matrix} -2&0\\0&-2\end{matrix}\right)$ respectively (same, up to a twist, as for $K3$s). For these two examples, we get  $\kappa_X=\frac{2^{21}}{3}$ and $2^{19}$ respectively. In contrast, as we will discuss below, the cubics that are conjectured to be irrational have $\kappa_X\le 1$. 
\end{remark}

No example of irrational cubic fourfold $X$ is known; we have nothing new to say in this direction. Thus, for all we know, it is possible that Question \eqref{questionmaxalg} is void. Nonetheless, we use it as motivation to answering the following more tangible question:
\begin{question}\label{question2}
Which  cubic fourfolds $X$, without associated $K3$ surfaces (in the sense of \cite{hassettT}),  maximize the algebraicity index $\kappa_X$? \end{question}

As discussed, the Harris--Hassett conjecture predicts that a cubic $X$ without an associated $K3$ surface is irrational.  Thus, the two questions  are related modulo a hard open  conjecture. Without touching on this conjecture, the goal of  our paper is to bring about a different perspective on it.  Namely, on a general cubic fourfold, 
 it is difficult to do geometry as there are no special algebraic cycles to build on. As we will see below, our answer $X$ to Question \ref{question2} is at the opposite end: e.g. $X$ contains $27$ planes. This is surprising, as  a cubic fourfold $X$ containing two disjoint planes not only has an associated $K3$, but it is in fact rational. 

\begin{definition}\label{defirrational}Let $X$ be a smooth cubic fourfold. Let $N_X:=H^{2,2}(X)\cap H^{4}(X,\bZ)$ be the lattice of algebraic cycles, and $T_X=(N_X)^\perp_{H^4(X,\bZ)}$ be the transcendental lattice.
We say that $X$ is {\it potentially irrational} if the transcendental lattice $T_X$ does not admit a primitive embedding into the (Tate twisted) $K3$ lattice $(E_8)^2\oplus U^3$. (In the literature,   {\it potentially irrational} means $X$ {\it does not have an associated $K3$ surface}.)
\end{definition}

Our main result is a full answer to Question \ref{question2}.
\begin{theorem}\label{mainthm}
Let $X$ be a cubic fourfold. 
\begin{itemize}
\item[(0)] If $X$ is potentially irrational, then $\kappa_X\le 1$.
\item[(1)] A general cubic fourfold $X$ containing an Eckardt point is potentially irrational with $\kappa_X=1$.
\item[(2)] Conversely, any  potentially irrational $X$ with $\kappa_X=1$  is a cubic fourfold with an Eckardt point. 
\end{itemize}
\end{theorem}

The cubics occurring in the theorem above were investigated in our paper \cite{LPZ} with G. Pearlstein and Z. Zhang. They arise in connection with our search for pseudo-cubics (i.e. Fano fourfolds with middle cohomology  of $K3$ type). As reviewed in Section \ref{sectlpz}, there are multiple geometric characterization of such cubics. The most relevant fact here is that they  admit an Eckardt point $p\in X$. This implies that  $X$ contains $27$ planes passing through $p$. It then follows that the lattice of primitive algebraic cycles is isometric to $E_6(2)$. In \cite{LPZ} we studied the moduli space of cubic fourfolds with an Eckardt point by viewing them as $E_6(2)$-marked cubics (in analogy to the $M$-polarized $K3$s of 
Dolgachev).

Theorem \ref{mainthm} (see \ref{pfitem1}, \ref{pfitem2}, and \ref{pfitem3} for proofs of items (1), (2), and (3) respectively) follows from the geometric considerations of \cite{LPZ}, and Nikulin's theory of lattices (\cite{nikulin,nikulinaut}). We have been inspired by the papers of Vinberg \cite{vinberg} and D. Morrison \cite{morrison} which touch on similar questions for $K3$s. An expert look at the lattices involved in our paper (i.e. $E_6(2)$ and $(D_4)^3\oplus U^2$) will bring into focus another  special features of the cubics $X$ studied in \cite{LPZ}. Namely,   their Fano variety $F(X)$ of lines admits an exotic anti-symplectic involution, i.e. an involution which is not induced (via deformations) from a $K3$ surface. In contrast, for symplectic involutions on hyper-K\"ahler fourfolds of $K3^{[2]}$ type, Mongardi \cite{mongardi} proved that they are all induced from $K3$ surfaces (N.B. this was generalized recently to the $K3^{[n]}$ type by Kamenova and Mongardi). For higher orders, Mongardi \cite{mongardithesis} and others produced examples of exotic symplectic automorphisms on hyper-K\"ahler fourfolds of $K3^{[2]}$ type (via constructions involving cubics). \begin{convention}\label{convention} The root lattices $ADE$ are positive definite. For a lattice $L$,  $L(a)$ denotes the lattice scaled by $a$. Depending on the context, the ``$K3$ lattice'' refers either to $(E_8(-1))^2\oplus U^3$ or $(E_8)^2\oplus U^3$.  
\end{convention}

\subsection*{Acknowledgement} We have been aware of the \cite{LPZ} cubics since the author's thesis in connection to the deformations of the $O_{16}$ singularity (the cone over a cubic surface). We thank R. Friedman and I. Dolgachev for many stimulating discussions on the subject. Many of the ideas in this paper crystallized  as consequence of discussions with G. Pearlstein and Z. Zhang around our joint work \cite{LPZ}. Finally, this paper was prompted by a brief chance discussion with L. Kamenova and A. Kumar on a related topic.

\section{General considerations on potentially irrational cubic fourfolds}
As a consequence of Torelli theorems, Question \ref{question2} is equivalent to the following lattice theoretic question. 

\begin{question}\label{questionembed}
Let $T$ be an even lattice of signature $(n,2)$ admitting a primitive embedding into the lattice $A_2\oplus (E_8)^2\oplus U^2$. When does $T$ admit a primitive embedding into the $K3$ lattice $E_8^2\oplus U^3$? 
\end{question}

\begin{remark}
A similar question was asked by D. Morrison \cite{morrison}: {\it Which $K3$ surfaces are Kummer?} Equivalently, {\it Let $T$ be a sublattice of the $K3$ lattice, when does $T$ embed into the Kummer lattice $U^3$}?
\end{remark}

We fix the following assumptions and notations consistent with the Introduction (except for dropping the index $X$, as the discussion here  is a purely lattice theoretic). 
\begin{notation}\label{bignotation}
Let $T$ be an even lattice of signature $(n,2)$ with $n\le 20$. Assume given a fixed primitive embedding 
$$T\hookrightarrow I_{21,2}$$
such that $N:=T^\perp$ contains a element $h$ with $h^2=3$ and such that $\langle h\rangle^\perp_{I_{21,2}}$ is an even lattice (we say $h$ is of even type). The choice of $h$ is unique up to the action of $O(I_{21,2})$, and assumed fixed throughout. An easy application of \cite{nikulin} shows that $\Lambda_0:=\langle h\rangle^\perp_{I_{21,2}}\cong A_2\oplus (E_8)^2\oplus U^2$. We let 
$$N_0:=\langle h\rangle^\perp_{N}.$$
The lattices $T$ and $N_0$ are mutually orthogonal lattices in $\Lambda_0$. Frequently, it is more convenient to consider $T$ as a sublattice of  $\Lambda_0$, and define $N_0$ as its orthogonal complement with respect to this embedding.  The extension to an embedding $T$ into $I_{21,2}$ (and similarly of $N_0$ to $N$) is unique. As before, we let 
\begin{eqnarray*}
\rho&:=&\rank N_0(=22-n),\\
d&:=&\det (T).
\end{eqnarray*}
(and $\kappa=\frac{2\rho}{d}$). For a lattice $L$, we denote by $A_L=L^*/L$ the discriminant group. If $L$ is an even lattice, $A_L$ comes endowed with a finite quadratic form $q_L:A_L\to \bQ/2\bZ$. The key invariant relevant in the discussion below is the discriminant group $A_T\cong A_N$ and the number of its invariant factors: 
\begin{equation}\label{defell}
\ell=l(A_T).
\end{equation}
\end{notation}

Nikulin's theory \cite{nikulin} provides a complete answer to embedding questions such as those raised in Question \ref{questionembed}. A key aspect of  Nikulin's results is that,  in most  cases, deciding that a lattice $L$ admits an embedding into some other lattice  (typically unimodular) reduces to comparing two integers. Specifically, in our situation, most of the cases of Question \ref{questionembed} can be decided by comparing the invariants $\ell$ and $\rho$.

\begin{proposition}\label{propcond}
With notation as above.  
\begin{itemize}
\item[i)] A necessary condition that $T$ embeds into the $K3$ lattice is $\ell\le \rho$;
\item[ii)] Conversely, if $\ell \le \rho-1$,  there exists a primitive embedding of $T$ into the $K3$ lattice.  Furthermore, if $\ell\le \rho-2$, the embedding is unique.
\end{itemize}
\end{proposition}
\begin{proof}
Assume $T$ embeds into the $K3$ lattice with orthogonal complement $M$.  Then, $A_T\cong A_M$. It follows 
$$\ell=l(A_T)=l(A_M)\le \min (\rank T, \rank M)=\min (\rho,22-\rho).$$
The converse follows from \cite[Cor. 1.12.3]{nikulin}.
\end{proof}

\begin{remark}\label{remmaxl}
The assumption $T$ embeds into the lattice $I_{21,2}$, automatically gives 
\begin{equation}\label{boundell}
\ell\le \min(\rho+1,22-\rho).
\end{equation}
If $\ell=\rho+1$, then $T$ does not admit an embedding into the $K3$ lattice (cf. Prop. \ref{propcond}(1)). Thus, the associated cubic fourfolds $X$ (with $T=T_X$) are potentially irrational in the sense of Def. \ref{defirrational}. An example is given by $X$ a general cubic fourfold, for which we have $T\cong A_2\oplus E_8^2\oplus U^2$,  $\rho=0$, and $\ell=1$.   Note that the algebraicity index is $\kappa_X=\frac{1}{3}$ for a general cubic fourfold (cf. \eqref{defkappa})\end{remark}

\begin{remark} Continuing the previous remark, let us note that for most cases $\ell\le \rho$, and that if this fails (i.e. $\ell=\rho+1$), then the algebraicity index $\kappa$ is small. First, note that  $N_0$ and $T$ are mutually orthogonal in $\Lambda_0$ and $A_{\Lambda_0}\cong \bZ/3$. It follows that either
$$A_T\equiv A_{N_0}/\bZ/3.$$
or  $A_T$ is a $\bZ/3$ extension of $A_{N_0}$. In the first case,  $\ell\le \rho$. While in the second case, the only possibility that $\ell=\rho+1$ is for the $3$-part of $A_T$ to have $\rho+1$ generators. Thus 
$$3^{\rho+1}\mid \det T=|A_T|.$$
which in turn gives $\kappa\le \frac{1}{3}$; the  equality holds only for  general  cubic fourfolds.
\end{remark}
 
In conclusion, the only interesting case for us is $\ell=\rho$. Furthermore, from Prop. \ref{propcond} and \eqref{boundell}, 
we note that if $\rho$ is big, then  $T$ embeds into the $K3$ lattice. 

\begin{proposition}\label{boundrho}  If $\rho\ge 11$, then $T$ admits a primitive embedding into the $K3$ lattice.
\end{proposition}
\begin{proof} $T$ has signature $(20-\rho,2)$. Note that $20-\rho\le 9$, and the proposition is precisely \cite[Cor. 2.10 and Rem. 2.11]{morrison} (based on \cite[Thm. 1.10.1]{nikulin}). 
\end{proof}

In conclusion, the maximal $\rho_X$ that can occur on a potentially irrational cubic fourfold $X$ is $10$ (N.B. we do not know if $\rho_X=10$ is achieved for some $X$). We expect a full classification of the cases $(\rho_X,d_X)$ for which $X$ is potentially irrational to be delicate as it involves the difficult case ($\ell=\rho$) of Nikulin's theory. For instance, we recall the following result of Hassett which answers the case $\rho=1$. We are also aware of some partial results for $\rho=2$ (e.g. \cite{AT}, \cite{auel}). 

\begin{theorem}[{Hassett \cite[Thm. 1.0.2]{hassettspecial}}] Let $X$ be a cubic fourfold and $d\equiv 0,2 \pmod 6$. Assume $\rho_X=1$ and $\det T_X=d$. Then $X$ is potentially irrational iff one of the following holds:
\begin{itemize}
\item[a)] $d\equiv 0 \pmod 9$
\item[b)] $d \equiv 0 \pmod 4$
\item[c)] $d$ has an odd prime factor $p\equiv - 1 \pmod 3$. 
\end{itemize}
\end{theorem}

 Our algebraicity index $\kappa=\frac{2^\rho}{d}$ is controlling the  growth of the rank $\rho$ relative to (log of) the size $d$ of the discriminant group (i.e. $\log_2\kappa=\rho-\log_2d$). It turns out that it is easy to classify the cases maximizing $\kappa$.  As a first step, we establish item (1) of  Theorem \ref{mainthm}:
\begin{corollary}\label{pfitem1}
If $X$ is a potentially irrational cubic fourfold. Then $\kappa_X\le 1$. Furthermore, if the equality holds, then $T_X$ is a $2$-elementary lattice. 
\end{corollary}
\begin{proof}
Let $T:=T_X$ be the transcendental cohomology of $X$. By assumption $T$ does not admit a primitive embedding  into the $K3$ lattice. By Proposition \ref{propcond} and Remark \ref{remmaxl}, $\ell\in \{\rho,\rho+1\}$. Since $\ell$ is the minimal number of generators of the abelian group $A_T$, and $d$ is the order of $A_T$, we conclude 
$$d\ge 2^\ell \ge 2^\rho.$$
The equality holds iff $A_T\cong (\bZ/2)^\rho$, which is to say $T$ is $2$-elementary.     
\end{proof}

\section{The $2$-elementary case}\label{case2elementary}
The $2$-elementary lattices have been recognized  by Nikulin \cite{nikulinaut} as playing a key role in the geometry of $K3$ surfaces; for instance they control the possible involutions on a $K3$ surface. 

\subsection{Nikulin's classification of $2$-elementary lattices}  Let $M$ be a $2$-elementary lattice, i.e.  $A_M\cong (\bZ/2)^l$. In addition to the obvious invariants of $M$: signature $(t_+,t_-)$ and $l=l(A_M)$, there is a parity invariant $\delta$ defined  by 
\begin{equation}\label{defdelta}
\delta:=\left\{ \begin{matrix} 0 & \textrm{ if $q_M$ takes values in $\bZ/2\bZ\subset \bQ/2\bZ$}\\
1 & \textrm{ else.}
\end{matrix}\right.
\end{equation}  

\begin{example}
$D_4$ and $U(2)$ are $2$-elementary lattices with $\delta=0$. While $A_1$ and $E_7$ have $\delta=1$. 
\end{example}

\begin{theorem}[{Nikulin, see \cite[Thm. 1.5.2]{dolgachevN}}]\label{thm2elementary}
The genus of an even $2$-elementary lattice $M$ is determined by the invariants $\delta$, $l$ and $(t_+,t_-)$.  If $M$ is indefinite, then the genus consists of one isomorphism class. An even $2$-elementary lattice $M$ with invariants $\delta$, $l$, and $(t_+,t_-)$ exists iff the following conditions are satisfied:
\begin{enumerate}
\item[0)] $l,t_+,t_-\ge 0$, $\delta\in\{0,1\}$;
\item[1)] $t_++t_-\ge l$;
\item[2)] $t_++t_-+l\equiv 0 \pmod 2$;
\item[3)] if $\delta=0$, then $t_+-t_-\equiv 0 \pmod 4$ (e.g. $U(2)$, $D_4$);
\item[4)] if $l=0$, then $\delta=0$, $t_+-t_-\equiv 0 \pmod 8$ (e.g. unimodular); 
\item[5)] if $l=1$, then $\left| t_+-t_-\right|\equiv 1 \pmod 8$ (e.g. $A_1$, $E_7$);
\item[6)] if $l=2$ and $t_+-t_-\equiv 4 \pmod 8$, then $\delta=0$ (e.g. $U(2)$, $D_4$); 
\item[7)] if $\delta=0$ and $l=t_++t_-$, then $t_+-t_-\equiv 0 \pmod 8$ (e.g. $U(2)$, $E_8(2)$). 
\end{enumerate}
\end{theorem}

\subsection{Potentially irrational cubics with $2$-elementary transcendental lattice} We are returning to our set-up (in particular, notations as in  \ref{bignotation}). In the context of Cor. \ref{pfitem1}, we assume that $T$ is a $2$-elementary lattice and that $T$ does admit a primitive embedding into the $K3$ lattice. 

\begin{theorem}\label{thm2elemcubic}
The only  $2$-elementary lattice $T$ of signature $(20-\rho,2)$ that embeds into $A_2\oplus (E_8)^2\oplus U^2$, but does not embed into the $K3$ lattice $(E_8)^2\oplus U^3$ is $T=(D_4)^3\oplus U^2$ (with $\sign=(14,2)$, $l=6$, and $\delta=0$).  
\end{theorem}
\begin{proof}
By Propositions \ref{propcond} and \ref{boundrho}, we can restrict to the case $\ell=\rho \le 10$ (here $l=\ell$). By Nikulin \cite[Thm. 1.12.2]{nikulin} $T$ admits an embedding into the $K3$ lattice iff there exists a $2$-elementary lattice $M$ (recall $A_M\cong A_T$) with invariants $\sign=(1,\rho-1)$, $l=\rho$, and $\delta=\delta_T$. Inspecting Theorem \ref{thm2elementary}, we see that if $\delta=1$, there is no further restriction on the existence of $M$ (or, more elementary, we can take $M=A_1\oplus(A_1(-1))^\rho$, corresponding geometrically to a double cover of $\bP^2$ branched in a sextic with $\rho(\le 10)$ nodes). Thus, $T$ as in the theorem should have additionally $\delta_T=0$. Under this assumption, the condition (3) of Theorem \ref{thm2elementary} gives $\rho\equiv 2\pmod 4$. There are $3$ values in the range of interest here: $\rho\in\{2,6,10\}$. The condition (7) of \ref{thm2elementary} then says that $M$ exists iff $\rho\equiv 2 \pmod 8$ (for the existence part, we can take $U(2)$ (i.e. hyperelliptic $K3$) and $U(2)\oplus E_8(2)$ (i.e. double cover of Enriques) respectively; the subtle part, where we need Thm. \ref{thm2elementary}, is the non-existence of $M$ for $\rho=6$). 

We conclude that the only potential case for $T$ is the $2$-elementary lattice with invariants $\sign=(14,2)$, $l=6$, and $\delta=0$. Since $T$ is indefinite, the isometry class of $T$ is determined by the invariants. Since $(D_4)^3\oplus U^2$ and $T$ have the same invariants, we conclude $T\cong  (D_4)^3\oplus U^2$.  The existence of an embedding of $T$ into the cubic lattice was established in \cite{LPZ}. In particular, we note that with respect to the embedding $T\hookrightarrow A_2\oplus (E_8)^2\oplus U^2$ considered in \cite{LPZ},
the orthogonal complement of $T$ is $E_6(2)$.  
\end{proof}

To complete the proof of item (3) of Theorem \ref{mainthm}, we need to conclude that the embedding of $T$ into $A_2\oplus (E_8)^2\oplus U^2$  is unique. By Nikulin's theory, this boils down to the statement that $E_6(2)$ is unique in its genus. This is probably well known, but for lack of a reference,  we provide a proof. 
\begin{lemma}\label{lemmauniquee6}
The only isometry class in the genus of $E_6(2)$ is $E_6(2)$. 
\end{lemma}
\begin{proof}
Assume $L$ is a lattice in the genus of $E_6(2)$. Thus, $L$ is even positive definite and 
$$A_L\cong A_{E_6(2)}\cong \bZ/3\times (\bZ/2)^6.$$
Furthermore, the finite quadratic form $q_L[2]$ on $A_L[2](:=3A_L)$ takes only integral values (i.e. it maps to $\bZ/2\bZ\subset \bQ/2\bZ$).

\begin{claim} The norm of any vector $L$ (a lattice in the genus of $E_6(2)$ is divisible by $4$.
\end{claim} 
\noindent{\it Proof (Claim).} If not, let $v\in L$ with $v^2\equiv 2\pmod 4$. Let $L':=\langle v\rangle_L^\perp$. We can view $L$ as an overlattice of $\langle v\rangle \oplus L'$. Following \cite[\S1.4]{nikulin}, there exists a gluing group $H$ with $H\subset A_{\langle v\rangle}\oplus A_{L'}$
such that 
$$A_{L}\cong H^\perp/H$$
where the orthogonal complement is taken in $A_{\langle v\rangle}\oplus A_{L'}$ with respect to the finite quadratic form $q_{\langle v\rangle}\oplus q_{L'}$ (and $H$ is totally isotropic w.r.t. to $q_{\langle v\rangle}\oplus q_{L'}$; in particular, $H\subset H^\perp$). Let us focus on the $2$-Sylow subgroups of the various  discriminant groups. First, since $v\equiv 2\pmod 4$, we get $A_{\langle v\rangle}[2]\cong \bZ/2$. On the other hand, $l(A_{L'}[2])\le 5$ and $l(A_L[2])=6$. We conclude that $H[2]$ is trivial.  Thus, we have a direct sum decomposition  
$$A_{L}[2]\cong A_{\langle v\rangle}[2]\oplus A_{L'}[2].$$
This clearly contradicts (since $A_{\langle v\rangle}[2]$ forces $\delta=1$) the fact that $A_{L}[2]$ takes only values in $\bZ/2\bZ\subset \bQ/2\bZ$.   \qed

Returning to the lemma, as a consequence of the claim,  we also have $v.w\equiv 0 \pmod 2$ for any $v,w\in L$ (e.g. $2\langle v,w\rangle=(v+w)^2-v^2-w^2$). 
We conclude that we can rescale $L$ by $\frac{1}{2}$. Thus $L(1/2)\cong E_6=E_6(2\cdot 1/2)$ (as $E_6$ is unique in its genus). The claim follows.
\end{proof}

\begin{theorem}\label{thmuniqe}
Let $T=(D_4)^3\oplus U^2$. Then there exists a unique primitive embedding of $T$ into the lattice $\Lambda_0=A_2\oplus (E_8)^2\oplus U^2$ (up to $O^*(\Lambda_0)$). This extends to a unique  embedding $T$ into the cubic lattice $I_{21,2}$. 
\end{theorem}
\begin{proof} The embedding $T\hookrightarrow \Lambda_0$ depends on the isometry class of $M=T^{\perp}_{\Lambda_0}$ and the choice of gluing subgroup $H\subset A_T\oplus A_M$ (with $H^\perp/H\cong A_{\Lambda_0}\cong \bZ/3$). Since $A_T\cong (\bZ/2)^6$, $A_{\Lambda_0}\cong \bZ/3$, and $M$ is positive definite of rank $6$, we conclude $A_M\cong A_{A_2}\oplus A_T$ and $M\cong H$. Since $E_6(2)$ and $T$ admit mutually orthogonal embeddings in $\Lambda_0$, we conclude that $M$ is in the same genus as $E_6(2)$ (as there is a unique choice of invariants). By Lemma \ref{lemmauniquee6}, $M\cong E_6(2)$ (and thus there is no choice for $M$). 

It remains to consider the choice of gluing group $H$. As noted, the natural projection $H\subset A_M\oplus A_T\to A_T$ induces an isomorphism $H\cong A_T(\cong (\bZ/2)^6$). Fixing an identification of the $2$-primary part in $A_M$ with  $A_T$, we see that the choices of $H$ induce automorphisms of $A_T (\cong \bZ/2)^6$, and in fact they are in $1$-to-$1$ correspondence with the isometries in $O(q_T)\cong W(E_6)$. Finally, since $T$ contains two hyperbolic summands $U$, the natural morphism $O(T)\to O(q_T)$ is surjective (\cite[Thm. 1.14.2]{nikulin}), concluding the proof of the fact that $T$ admits a unique embedding into $\Lambda_0$ modulo $O^+(\Lambda_0)$.  Since $A_{\Lambda_0}\cong \bZ/3$ is identified with the $3$-primary part of $A_M$, and since the gluing group $H$ only involves the $2$-primary part,  all choices involved (and identifications made)  respect the discriminant $A_{\Lambda_0}$. Thus, we get uniqueness modulo $O^*(\Lambda_0)$. 

The lattice $\Lambda$ is obtained by glueing $\Lambda_0$ and $\langle h\rangle$ (with $h^2=3$) along $H\cong \bZ/3$. Arguing as above,   we can chose the gluing vector $v$ (i.e. $v$ such that $v+h$ is divisible by $3$ in $\Lambda$) in $M$. The claim follows.
\end{proof}

At this point, we conclude that there is only one class of maximally algebraic potentially irrational cubics. 

\begin{corollary}\label{pfitem3}
The Zariski closure of the locus of potentially irrational cubic fourfolds $X$  with $\kappa_X=1$ is an \underline{irreducible} $14$ dimensional algebraic subvariety in the moduli space $\calM$ of cubic fourfolds 
\end{corollary}
\begin{proof}
By Corollary \ref{pfitem1} and Theorem \ref{thm2elemcubic}, we conclude that the transcendental lattice $T_X$ of a cubic fourfold as in the statement of the corollary is isometric to $(D_4)^3\oplus U^2$. By Theorem \ref{thmuniqe}, $T_X$ admits a unique embedding into the primitive cubic lattice $\Lambda_0$. Thus, we can uniquely (up to monodromy)  mark a transcendental lattice $T\cong (D_4)^3\oplus U^2$ inside $H^4(X,\bZ)$. It follows that  
 the locus of cubics with periods belonging to $T$ is an irreducible  subvariety $Z$ (of codimension $6$) of the period domain $\calD/\Gamma$. 
 
 Recall $\calD/\Gamma$ is a quasi-projective variety (Baily-Borel Theorem), and that there exists a period map 
 $$\calP:\calM\to \calD/\Gamma\setminus (\calC_2\cup \calC_6),$$
 that is an isomorphism of quasi-projective varieties (cf. \cite{voisin}, \cite{cubic4fold}, \cite{lcubic}). We conclude that $\calP^{-1}\left(Z\setminus (\calC_2\cup \calC_6)\right)\subset \calM$ is an irreducible subvariety in the moduli space $\calM$ of cubic fourfolds parameterizing (the closure of the locus of) maximally algebraic potentially irrational cubic fourfolds. 
 
 A final point here is that  the divisors $\calC_2$ and $\calC_6$ that are missing from the image of the period map (cf. \cite{hassettspecial}, \cite{cubic4fold}, \cite{lcubic}) have (from a lattice theoretic point of view) associated $K3$ surfaces, of degree $2$ and $6$ respectively. This implies that the locus $Z$ will not be completely contained in $\calC_2\cup \calC_6$. Thus $\calP^{-1}\left(Z\setminus (\calC_2\cup \calC_6)\right) \neq \emptyset$ and of expected dimension.
 \end{proof}

\begin{remark}\label{remindex2}
As always with cubic fourfolds (esp. those containing planes), the (potential) irrationality rest on a subtle index $2$ issue. Specifically, if instead of  $T_X=(D_4)^3\oplus U^2\cong D_4\oplus E_8\oplus U(2) \oplus U(2)$ we are considering  the index $2$ overlattice $D_4\oplus E_8\oplus U(2)\oplus U$, then there exists an associated $K3$ surface $S$ with Neron-Severi lattice $M=D_4\oplus U(2)$. A geometric realization for $S$ can be obtained via the double cover of $\bP^2$ branched in a quintic and a line. There is a close geometric relationship between $X$ (with $T_X$ as above) and $S$. Both cases arise from the study of certain singularities ($O_{16}$ and $N_{16}$ respectively; see \cite{raduthesis}). 
\end{remark}

\section{Cubic fourfolds with an Eckardt point}\label{sectlpz}
The only remaining issue is to provide a geometric meaning for cubic fourfolds $X$ with transcendental cohomology $T_X\cong (D_4)^3\oplus U^2$. The content of \cite{LPZ} is to analyze the moduli space of cubics $X$ with an Eckardt point $p$, which turn out to have transcendental lattice precisely $(D_4)^3\oplus U^2$. In view of Corollary \ref{pfitem3}, this concludes the proof of Theorem \ref{mainthm}.

The cubic fourfolds considered in \cite{LPZ} arise from multiple different perspectives, and have many special properties. Let us give a brief overview of the situation. First, motivated by the study of the singularity $O_{16}$ (see \cite{raduthesis}), we can consider cubic pairs $(Y,S)$ consisting of a cubic threefold and a hyperplane section $S$ of $Y$. By a construction analogous to \cite{ACT}, we can encode such a pair $(Y,S)$ into a cubic fourfold $X$ together with an involution $\iota$ fixing a hyperplane. Explicitly,  we can choose coordinates on $\bP^5$ such that 
\begin{equation}\label{eqlpz}
X=V(f_3(x_0,\dots,x_4)+x_0x_5^2),
\end{equation}
where $f_3$ is a general cubic form (with $Y=V(f_3)\subset \bP^4$,  and $S=V(f_3,x_0)$). The cubic $X$ is characterized by the fact that it carries an involution $\iota$ (explicitly $x_5\to -x_5)$ fixing a hyperplane (here $V(x_5)$) and a point $p$ (here $p=(0,\dots,0,1)$).  It turns out that $p$ is an Eckardt point for $X$ (i.e. $X\cap T_pX$ is the cone over the cubic surface $S$), and that this can be used to give a different geometric characterization of $X$. Finally, $X$ also occurs as a smooth birational model for a general degree $6$ weighted hypersurface in $W\bP(1,2,2,2,2,3)$ (which occurs in the analysis of the four-dimensional analogue of Reid's list of $95$ $K3$ surfaces that are weighted hypersurfaces; this was our original motivation for \cite{LPZ}). All these geometric characterizations are shown to be equivalent in \cite{LPZ}. The relevant bit here is the following:

\begin{theorem}[{\cite{LPZ}}]
For a general cubic fourfold $X$ as in \cite{LPZ}, the following hold:
\begin{itemize}
\item[i)] $X$ contains $27$ planes $\Pi_i$ passing through the Eckardt point $p$; \item[ii)] The primitive algebraic cohomology $N_0\cong E_6(2)$ (spanned by classes $[\Pi_i]-[\Pi_j]$);
\item[iii)] The transcendental cohomology of $X$ is $T\cong (D_4)^3\oplus U^2$. 
\end{itemize}
\end{theorem}

\begin{corollary}\label{pfitem2}
A  general cubic fourfold $X$ containing an Eckardt point is maximally algebraic potentially irrational. Conversely, any maximally algebraic potentially irrational $X$ is a cubic with a single Eckardt point. \qed
\end{corollary}

 \bibliography{maxcubic}
\end{document}